\documentclass[11pt, reqno]{amsart}

\usepackage[margin=1.2in]{geometry}
\usepackage{amssymb,amscd,amsmath}
\usepackage{bbm}
\usepackage{nicematrix}
\usepackage[colorlinks=true,citecolor=blue,linkcolor=blue]{hyperref}
\usepackage[dvipsnames]{xcolor}
\usepackage{tikz}
\usetikzlibrary{arrows,matrix,positioning,fit}
\usepackage{multicol}

\usepackage[nameinlink]{cleveref} 
\usepackage{verbatim}

\definecolor{color1}{HTML}{0077FF}
\definecolor{color2}{HTML}{FF9500}
\definecolor{color3}{HTML}{00A878}
\definecolor{color4}{HTML}{D83B8D}

\numberwithin{equation}{section}

\newtheorem*{introthm*}{Theorem}
\newtheorem{theorem}{Theorem}[section]
\newtheorem{lemma}[theorem]{Lemma}
\newtheorem{proposition}[theorem]{Proposition}
\newtheorem{corollary}[theorem]{Corollary}
\newtheorem*{claim*}{Claim}

\theoremstyle{definition}
\newtheorem{definition}[theorem]{Definition}
\newtheorem{construction}[theorem]{Construction}

\newtheorem{remark}[theorem]{Remark}
\newtheorem{example}[theorem]{Example}
\newtheorem*{acknowledgement}{Acknowledgements}

\DeclareMathOperator{\gr}{gr}

\newcommand{\NN}{{\mathbb N}}
\newcommand{\QQ}{{\mathbb Q}}
\newcommand{\RR}{{\mathbb R}}
\newcommand{\ZZ}{{\mathbb Z}}

\def\F{{\mathcal F}}

\def\ba{{\bf w}}
\def\c{{\bf c}}
\def\x{{\bf x}}
\def\z{{\bf t}}
\newcommand{\kk}{{\mathbbm k}}
\def\mm{{\mathfrak m}}
\DeclareMathOperator{\gap}{gap}
\DeclareMathOperator{\HS}{HS}

\DeclareMathOperator{\lcm}{lcm}

\title{Weighted Veronese Rings via Convex Semigroups}

\author[Chase]{Bek Chase}
\address{University of Central Arkansas, Department of Mathematics, 201 Donaghey Ave, Conway, AR 72035, USA}
\email{rchase@uca.edu}

\author[Fiorindo]{Luca Fiorindo} 
\address{Università degli Studi di Genova, Dipartimento di Matematica, Dipartimento di Eccellenza 2023-2027, Via Dodecaneso 35, 16146 Genova, Italy}
\email{luca.fiorindo@edu.unige.it}

\author[Holleben]{Thiago Holleben}
\address{MPI für Mathematik in den Naturwissenschaften,
Inselstraße 22, 04103, Leipzig,
Germany}
\email{holleben@mis.mpg.de}

\author[Marangone]{Emanuela Marangone}
\address{Università degli Studi di Genova, Dipartimento di Matematica, Via Dodecaneso, 35, 16146 Genova, Italy \newline
\indent CIMAT - Centro de Investigación en Matemáticas, Valenciana, 36023 Guanajuato, Mexico}
\email{emanuela.marangone@edu.unige.it, emanuela.marangone@cimat.mx}

\author[Nguy$\tilde{\text{\^E}}$n]{Th\'ai Th\`anh Nguy$\tilde{\text{\^E}}$n}
\address{University of Dayton, Department of Mathematics and Statistics,
	300 College Park, Dayton, OH 45469, USA \newline
\indent University of Education, Hue University, 34 Le Loi St., Hue, Viet Nam}
\email{tnguyen5@udayton.edu}

\author[Seceleanu]{Alexandra Seceleanu}
\address{University of Nebraska-Lincoln, Department of Mathematics, 203 Avery Hall, Lincoln NE 68588, USA}
\email{aseceleanu@unl.edu}

\author[Singh]{Srishti Singh}
\address{University of Tennessee-Knoxville, Mathematics Department, 227 Ayres Hall
1403 Circle Drive,
Knoxville TN 37996, USA}
\email{ssingh63@utk.edu}


\begin{document}

\begin{abstract}
    We determine several properties of two-dimensional normal affine semigroup rings, and in particular of weighted Veronese rings, including determinantal presentations, Gr\"obner bases, graded Hilbert series and graded Betti numbers, the structure of their associated graded rings, and their Koszul property. We give examples in higher dimensions illustrating that the first and last properties may fail. Our approach leverages convex monomial ideals as introduced in \cite{ConvexIdeals}, which give rise to convex semigroups.
\end{abstract}

\maketitle


\section*{Introduction}

Let $R=\bigoplus_{i\geq 0} R_i$ be an $\mathbb{N}$-graded ring. For a fixed integer $d\geq 1$,
the \emph{$d$-th Veronese subring} of $R$ is defined as
\[
V^{\langle d\rangle}(R)=\bigoplus_{i\geq 0} R_{di}.
\]
When $R$ is a standard graded polynomial ring over a field $\kk$, that is,
$R=\kk[x_1,\ldots,x_n]$ with $\deg(x_i)=1$, we denote the corresponding Veronese ring by
$V_{n,d}$. 

Since the work of Giuseppe Veronese (1854–1917), Veronese varieties have played, and continue to play, a central role in algebraic geometry, invariant theory, and representation theory.
In geometry the $d$-th Veronese variety in $\mathbb{P}^n$ is realized as the image of the complete linear system
$$v_{n+1,d}:\mathbb{P}^n\xrightarrow{|\mathcal{O}(d)|}\mathbb{P}^N,$$
where $\mathcal{O}$ is the structure sheaf of $\mathbb{P}^n$. Its coordinate ring is isomorphic to $V_{n+1,d}$. We shall describe shortly how the algebraic definition stated above also arises naturally from invariant theory and the study of semigroup rings.

The standard graded Veronese rings $V_{n,d}$ enjoy a rich collection of algebraic
properties. They are generated as $\kk$-algebras by the monomials of degree $d$, and admit
presentations of the form
\[
V_{n,d}\cong \kk[t_1,\ldots,t_N]/I_{n,d},
\]
where the variables $t_i$ correspond to degree-$d$ monomials in $R$. The defining ideal
$I_{n,d}$ is quadratic, binomial, and determinantal, generated by the $2\times2$ minors of
a suitable matrix \cite[Exercise~6.10]{Eis}. Moreover, these minors form a Gr\"obner basis
\cite{CDR}, from which it follows that $V_{n,d}$ is Koszul \cite{Anick}; see also \cite{BarcanescuManolache} for a different approach. The rings $V_{n,d}$
are also normal and hence Cohen--Macaulay \cite{Hoc}. In the case $n=2$, the minimal free
resolution of $V_{n,d}$ is given by the Eagon--Northcott complex and is therefore
$2$-linear. For $n\geq 3$, the structure of minimal free resolutions is significantly more
subtle and is closely related to the Green--Lazarsfeld property $N_p$, whose full
classification remains open.

In this paper we replace the standard grading by a non-standard one. Specifically, we
consider a polynomial ring $R=\kk[x_1,\ldots,x_n]$ equipped with a positive grading
$\deg(x_i)=w_i$, and write $V_{\ba,d}$ for the $d$-th Veronese subring, where $\ba=(w_1,\ldots,w_n)$. 
In this more general
setting, it is natural to ask to what extent the structural properties of the standard
Veronese rings persist.

When $\kk$ contains a primitive $d$-th root $\xi$ of one, $V_{\ba,d}$ can be realized from an invariant-theoretic perspective as the ring of invariants of $R$ under the natural action of the cyclic group  of $d$-th roots of unity acting by graded scaling $x_i\mapsto \xi^{w_i}x_i$. Therefore the affine spectrum of a weighted Veronese ring realizes a cyclic quotient singularity.
The rings $V_{\ba,d}$ were studied as rings of invariants by Harris and Wehlau in \cite{HW} for $n=2$ and for $n=3$ under the additional assumption that $d$ divides $\sum_{i=1}^3 w_i$. They described minimal generators and relations for these rings of invariants, proved that the lead terms of the relations are quadratic, and  described the graded Betti numbers for their minimal free resolution over $R$ (see also \cite{CHW}).

In this paper we revisit two-dimensional weighted Veronese rings from a different perspective, focusing on identifying their associated graded rings. Our approach is structural: we introduce the class of {\em convex semigroups}, cf.~\Cref{def: convex semigroup ring}, motivated by convex monomial ideals studied by Herzog-Qureshi-Mohammadi in \cite{ConvexIdeals}. In our terminology, a convex semigroup is a two-dimensional semigroup whose generators, viewed as lattice points, form a convex sequence, cf.~\Cref{def:convex}.   
We identify the associated graded rings of convex semigroups as  special fiber rings of convex monomial ideals. This new connection is the main tool that we develop and leverage in this paper. Utilizing the degeneration of a convex semigroup ring to  its associated graded ring, we show in \Cref{thm: convex semigroup} that convex semigroup rings admit determinantal
presentations, are normal and Cohen--Macaulay, and their associated graded rings are Koszul.

This result admits both applications and extensions:  the main application developed in the current work is to two-dimensional Veronese rings, which we show are convex semigroups in \Cref{lem.ConvexGens}. However, since all two-dimensional normal affine semigroup rings are  isomorphic to cyclic quotient singularities cf.~\Cref{prop: 2Dnormal_is_Veronese}, many of our results on convex semigroups and weighted Veronese rings are applicable to all two-dimensional normal affine semigroup rings. 

We now state our main result, \Cref{thm: 2-dimVeronese}.  For the particular case of weighted Veronese rings, items (2), (3) have been established  using different methods in \cite{HW}.

\begin{introthm*}\label{introthm}
Let  $S$ be a two-dimensional normal affine semigroup ring.
Then
\begin{enumerate}
\item $S$ has a determinantal presentation $S=\kk[t_1,\ldots, t_s]/J$,
where $J$ is the ideal of $2\times 2$ minors of a matrix  whose entries are nonconstant monomials in the variables $t_1, \ldots, t_s$;
\item  the minimal generators of $J$ form a Gr\"obner basis under the lexicographic order induced by $t_1>t_2>\cdots>t_n$;
\item the total Betti numbers of $S$ over $\kk[t_1,\ldots, t_s]$ depend only on the number $s$ of minimal  monomial $\kk$-algebra generators of $S$
\[
\beta_i^R(S)= i\binom{s-1}{i+1};
\]
\item $S$ is a normal, Cohen--Macaulay, Koszul algebra (i.e., $\gr_\mm S$ is Koszul).
\end{enumerate}
If additionally $S$ is a convex semigroup ring (in particular if $S=V_{\ba,d}$) and $I=(x^{a_i}y^{b_i} : 1\leq i\leq s)$ is the ideal generated by the minimal monomial $\kk$-algebra generators of $S$ then
\begin{enumerate}
   \item[(5)] the associated graded ring $\gr_\mm S$ is isomorphic to the special fiber ring $\F(I)$;
    \item[(6)] suppose that $\deg(x^{a_i}y^{b_i})=d_i$, then the graded Betti numbers, expressed using the notation in \eqref{eq: gap}, are
    \[
    \beta_{i,j}(S_\c)=\sum_{\substack{A\subseteq [s], |A|=i+1, \\\sum_{\ell\in A} d_\ell=j}} \gap(A);
    \]
    \item[(7)] the Hilbert series of $S$ depends only on the sequence $(d_i)_{i=1}^s$ and is given by 
    \[\HS(z,S)=\frac{1+\sum_{A\subseteq [s]}(-1)^{|A|+1}\gap(A)\,z^{\sum_{k\in A}d_k}}{\prod_{j=1}^s(1-z^{d_j})}.
    \]
\end{enumerate}
\end{introthm*}

Determinantal equations for classes of algebraic varieties have been produced under the assumption of a sufficiently ample embedding by Mumford \cite{Mumford} and Sidman--Smith \cite{SidmanSmith} among others. By contrast, item (1) in the theorem above provides determinantal equations for two-dimensional normal affine toric varieties without additional assumptions on the embedding. Moreover, item (5) can be interpreted as a degeneration of such a variety to a union of rational normal curves and projective lines in view of \Cref{thm: convex fiber ring}~(1). Finally, item (4) confirms in a stronger form a conjecture of Davis--Erman--Martinova
\cite[Conjecture~1.5]{DavisMartinova} regarding the Koszul property of weighted Veronese rings in the two-dimensional case. This conjecture has been refuted in dimension at least four in \cite{Bruce}, while the case of three-dimensional weighted Veronese rings is still open. In this paper we employ the terminology in \cite{HRW} to say that a not necessarily standard graded ring is Koszul provided its associated graded ring with respect to the homogeneous maximal ideal is Koszul in the usual sense.

Beyond dimension two, the behavior changes markedly. We show in \Cref{prop: V_{(3,4,5),6}} that higher-dimensional
weighted Veronese rings need not admit determinantal presentations and in \Cref{prop: nonKoszul} that their associated graded rings may fail to be
Koszul. Explicit families of examples illustrate both phenomena. However, we note in \Cref{rem: normal higher dim} that higher-dimensional weighted Veronese rings are still normal and hence Cohen-Macaulay.

The paper is organized as follows. 
In \Cref{s: convex ideals} we review the theory of convex monomial ideals and we prove new formulas \Cref{prop: betti fiber} and \Cref{prop:closedformulamonomial} for the invariants of their fiber rings with respect to a weighted grading. These constitute the basis for items (3), (6) and (7) of our main theorem. 
In~\Cref{s: convex semigroups}, we develop the theory of convex semigroups, which we use in~\Cref{s: 2dim}  to analyze two-dimensional weighted Veronese rings, culminating in establishing the properties listed above.  Finally, in~\Cref{s: higher dimensional Veronese rings}, we exhibit examples showing the
failure of determinantal and Koszul properties in higher dimensions.

\begin{acknowledgement}
This work began at the Apprenticeship Workshop of the 2025 Fields Institute semester program
in Commutative Algebra and its Applications. We warmly thank the organizers of the workshop
and program for bringing us together, and the Fields Institute for supporting our
participation. Fiorindo is a member of INdAM--GNSAGA and gratefully acknowledges its
support. Marangone acknowledges support from the Pacific Institute for the
Mathematical Sciences. Marangone also acknowledges the European Union’s Horizon Europe research and innovation program under the Marie Skłodowska-Curie Funding Program (Project PrIMes, Grant Agreement No. 101277460). Nguy$\tilde{\text{\^e}}$n is partially supported by the NAFOSTED under the grant number 101.04-2023.07 and the AMS-Simons Travel Grant. Seceleanu is partially supported by NSF DMS--2401482. 

We thank
Daniel Erman and Maya Banks for insightful comments,
David Wehlau for bringing the works \cite{HW, CHW} to our attention, 
and Dennis Belotserkovskiy for suggesting~\Cref{prop: nonKoszul}~(2).
\end{acknowledgement}

\section{Background}

\subsection{Convex ideals}\label{s: convex ideals}

We recall the notion of convex ideals introduced in \cite{ConvexIdeals}.

\begin{definition}\label[definition]{def:convex}
A sequence $\{c_i=(a_i,b_i)\}_{i=1}^s$ of integer vectors in $\NN^2$, ordered so that
$
a_1\le a_2\le \cdots \le a_s,
$
is called \emph{convex} provided that for each $2\le i\le s-1$ the inequality
\begin{equation}\label{eq: convex}
   2c_i\le c_{i-1}+c_{i+1} 
\end{equation}
holds componentwise. We call $c_i$ a \emph{corner point} of the convex sequence if $i=1$, $i=s$, or if  strict inequality holds in \eqref{eq: convex}.
A useful property of convex sequences, proved in~\cite[Lemma~1.1(b)]{ConvexIdeals}, is that for all indices $i\le j$ and all integers $k\ge 0$ such that the expressions are defined, one has
\begin{equation}\label{eq: convex property}
c_i + c_j \le c_{i-k} + c_{j+k}.
\end{equation}
and the inequality in~\eqref{eq: convex property} is strict whenever there exists at least one corner point of the sequence with index strictly between $i-k$ and $j+k$.

In~\cite{ConvexIdeals}, a monomial ideal $I\subset \kk[x,y]$ is called a \emph{convex ideal} if 
$
I=\left(x^{a_i}y^{b_i}: 1\leq i\leq s \right),
$
where $c_i=\{(a_i,b_i)\}_{i=1}^s$ is a convex sequence. Whenever $c_i$ is a corner point of this sequence, we call $c_i$ a corner point of $I$.
\end{definition}

For a visual representation, see the convex sequence depicted in Figure~\ref{fig:veronese}.

A remarkable theorem concerning the fiber cone of a convex monomial ideal is proved in~\cite{ConvexIdeals}. If $I$ is a homogeneous ideal of a standard graded ring $Q$ with maximal homogeneous ideal $\mathfrak m_Q$, the \emph{fiber cone} of $I$ is the standard graded ring
\begin{equation}\label{eq: fiber cone}
\mathcal F(I)=\bigoplus_{i\ge 0} \frac{I^i}{\mathfrak m_Q I^i}.
\end{equation}
If $I=(u_1, \ldots, u_s)$, then $\mathcal F(I)$ admits a natural surjection
\begin{equation}\label{eq: phi}
\varphi: \kk[t_1, \ldots, t_s]\to \mathcal F(I), \qquad \varphi(t_i)=\overline{u_i}\in I/\mathfrak m_Q I.
\end{equation}
When $I$ is a monomial ideal that is $\NN^n$--graded, the fiber cone inherits an $\NN^n$--grading, and hence admits a natural $\NN\times \NN^n$--grading.

\begin{theorem}[{\cite[Theorem~2.10 and Proposition~2.7]{ConvexIdeals}}]\label{thm: convex fiber ring}
Let $I$ be a convex ideal with minimal monomial generating set $I =(m_1, \ldots, m_s)$, and let $\mathcal F(I)=\kk [t_1, \ldots, t_s]/L$ be its fiber cone. Assume that $\{c_{j_1},\ldots, c_{j_\ell} \}$, with
$
1 = j_1 < j_2 < \cdots < j_\ell = s,
$
is the set of corner points of $I$.

For $1\le k<\ell$, consider the $2\times (j_{k+1}-j_k)$ matrix
\begin{equation}\label{eq: matrix M}
M_k=\begin{bmatrix}
t_{j_k} & t_{j_k+1} & \cdots &t_{j_{k+1}-1}\\
t_{j_k+1} & t_{j_k+2} & \cdots & t_{j_{k+1}}
\end{bmatrix},
\end{equation}
and let $\mathcal B$ be the union of the sets of all $2\times2$ minors of the matrices $M_k$.

Let $<$ denote the lexicographic order induced by $t_1>\cdots>t_s$, and let $\mathcal M$ be the set of monomials $t_i t_j$ such that
$
1 \le i < j-1 \le s-1
$
and $t_i t_j$ is not the initial monomial of any element of $\mathcal B$ with respect to $<$. Then:
\begin{enumerate}
\item $L=(\mathcal B)+(\mathcal M)$ and
\[
\operatorname{in}_<(L)=(t_i t_j \mid 1 \le i < j-1 \le s-1).
\]
Moreover, the minimal generators of $L$ form a Gr\"{o}bner basis with respect to $<$.
\item The fiber cone $\mathcal F(I)$ is a reduced Cohen--Macaulay Koszul algebra.
\item The Hilbert series of $\mathcal F(I)$ is
\[
HS_{\mathcal F(I)}(z)=\frac{1 + (s-2)z}{(1-z)^2}.
\]
\end{enumerate}
\end{theorem}

We now add an additional homological property of $\mathcal F(I)$ to the list above.

\begin{proposition}\label[proposition]{prop: betti fiber}
Under the hypotheses of~\Cref{thm: convex fiber ring}, the graded Betti numbers of the fiber cone satisfy
\[
\beta_{0,0}(\mathcal F(I))=1,
\]
and for $i\ge 1$,
\begin{equation}\label{eq: betti}
\beta_{i,i+1}(\mathcal F(I))=i\binom{s-1}{i+1}, \qquad \beta_{i,j}(\mathcal F(I))=0 \text{ for } j\neq i+1.
\end{equation}
\end{proposition}

\begin{proof}
By~\Cref{thm: convex fiber ring}, the initial ideal of $L$ with respect to the lexicographic order is
\[
J=(t_i t_j \mid 1 \le i < j-1 \le s-1).
\]
This is the Stanley--Reisner ideal of a path simplicial complex $\Delta$ on $s$ vertices with $s-1$ edges connecting consecutive vertices. Since $\Delta$ is a connected graph, Reisner's criterion~\cite[Theorem~1]{R1976} implies that $J$ is Cohen--Macaulay. Moreover, as $\Delta$ is a tree and hence chordal, Fr\"{o}berg's theorem~\cite[Theorem~1]{F1990} implies that $J$ has a linear resolution.

Let $Q=\kk[t_1, \ldots, t_s]$ with the standard grading. Then $Q/J$ is a $2$--dimensional Cohen--Macaulay $Q$--module with a pure minimal free resolution whose shifts are given by $d_0=0$ and $d_i=i+1$ for $i\ge1$. By the Herzog--K\"{u}hl equations~\cite{HK1984} (see also~\cite[Section~1.4]{BSsurvey}), there exists $t\in\QQ$ such that
\[
\beta_{i,d_i}(Q/J)=(-1)^i t\prod_{k\ne i}\frac{1}{d_k-d_i}.
\]
To compute $t$, note that
\[
1=\beta_{0,0}(Q/J)=t\prod_{k\ne0}\frac{1}{d_k}=\frac{t}{(s-1)!},
\]
so $t=(s-1)!$. Therefore, for $i\ge1$ we obtain
\begin{align*}
\beta_{i,i+1}(Q/J)
&=(-1)^i (s-1)! \prod_{k\ne i} \frac{1}{k-i}\
&=\frac{(s-1)!}{(i+1)\prod_{0<k<i}(i-k)\prod_{i<k<s-1}(k-i)}\
&= i\binom{s-1}{i+1}.
\end{align*}

By~\cite{Peeva}, the minimal $Q$--free resolution of $\mathcal F(I)=Q/L$ is obtained from that of $Q/J$ by consecutive cancellations. Since the resolution of $Q/J$ is linear, no such cancellations can occur. Hence
\[
\beta_{i,i+1}(\mathcal F(I))=\beta_{i,i+1}(Q/J)=i\binom{s-1}{i+1},
\]
which completes the proof.
\end{proof}

In~\Cref{s: 2dim} we will be interested in viewing the monomial ideal
$
(t_i t_j \mid 1 \leq i < j-1 \leq s-1)
$
as a homogeneous ideal of a weighted polynomial ring. To this end, we determine the
$\NN^s$-graded structure of its minimal free resolution, as well as its Hilbert series
with respect to a given non-standard grading.

Let $A=\{n_1,\dots,n_i\}\subset [s]$ with $n_1<\cdots<n_i$. We define the
\emph{gap number} of $A$ by
\begin{equation}\label{eq: gap}
\gap(A)=\bigl|\{\,\ell \mid n_{\ell+1}-n_\ell>1\,\}\bigr|.
\end{equation}

\begin{proposition}\label[proposition]{prop:closedformulamonomial}
Let $Q=\kk[t_1,\ldots,t_s]$ and
$
J=(t_i t_j \mid 1 \leq i < j-1 \leq s-1).
$
Then the $(i-1)$-st free module in a minimal free resolution of $Q/J$ is
\begin{equation}\label{eq: multiBetti}
F_{i-1}
=\bigoplus_{A\subseteq [s],|A|=i} Q\!\left(-\prod_{j\in A} t_j\right)^{\gap(A)}.
\end{equation}
In particular, if $\deg(t_i)=d_i$, then the Hilbert series of $Q/J$ is given by
\begin{equation}\label{eq: HS 2dim}
\HS(z,Q/J)
=
\frac{1+\sum_{A\subseteq [s]}(-1)^{|A|+1}\gap(A)\,z^{\sum_{k\in A}d_k}}
{\prod_{j=1}^s(1-z^{d_j})}.
\end{equation}
\end{proposition}

\begin{proof}
The ideal $J$ is the Stanley--Reisner ideal of the path graph $P_s$ on $s$ vertices.
Consequently, the only (nonempty) induced subcomplexes of $P_s$ with nontrivial reduced
homology are the disconnected sets. Such subcomplexes are in bijection with subsets
$A\subseteq [s]$ satisfying $\gap(A)>0$.

Let $m=\prod_{j\in A}t_j$ be a squarefree monomial of degree $|A|$. By Hochster’s formula,
\[
\beta_{|A|-1,m}(Q/J)
=\dim_\kk \widetilde H_0(P_s[A];\kk)
=\gap(A),
\]
where $P_s[A]$ denotes the induced subcomplex of $P_s$ on the vertex set $A$. This yields
the claimed description of the free modules in the minimal resolution whence formula \eqref{eq: HS 2dim} for the Hilbert series follows.
\end{proof}

\begin{example}
Let $s=4$, $Q=\kk[t_1,\dots,t_4]$ with $\deg(t_i)=d_i$, and
$
J=(t_1t_3,t_1t_4,t_2t_4).
$
The subsets $A$ of $[4]$ with nonzero gap number each having $\gap(A)=1$ are
\[
\{1,3\},\ \{1,4\},\ \{2,4\},\ \{1,2,4\},\ \{1,3,4\}.
\]
Therefore, \eqref{eq: HS 2dim} gives
\[
\HS(z,Q/J)
=
\frac{
1
- z^{d_1+d_3}
- z^{d_1+d_4}
- z^{d_2+d_4}
+ z^{d_1+d_2+d_4}
+ z^{d_1+d_3+d_4}
}{
(1-z^{d_1})(1-z^{d_2})(1-z^{d_3})(1-z^{d_4})
}.
\]
\end{example}

\section{Convex semigroup rings}\label{s: convex semigroups}

 \Cref{def:convex} formalizes the notion of a convex sequence of lattice points
$\c=\{c_i\}_{i=1}^s\subset \NN^2$. Motivated by this, we introduce the following terminology.

\begin{definition}\label[definition]{def: convex semigroup ring}
A \emph{convex semigroup} is a subsemigroup of $\NN^2$ generated by a convex sequence
$\c\subset \NN^2$. A \emph{convex semigroup ring} is the $\kk$-subalgebra of $\kk[x,y]$
generated by the monomials corresponding to $\c$, namely
\begin{equation}
\label{eq: Sc}
S_{\c}=\kk\bigl[x^{a_{i}}y^{b_{i}}\mid \c=\{(a_{i},b_{i})\}_{i=1}^s\bigr]
\subset \kk[x,y].
\end{equation}
\end{definition}

\begin{example}\label[example]{ex: convex}
The weighted Veronese ring
\[
V_{(1,29),47}
=\kk[x^{47},x^{18}y,x^7y^3,x^3y^8,x^2y^{21},xy^{34},y^{47}]
\]
is a convex semigroup ring.~\Cref{fig:veronese} illustrates the exponent vectors of its
generators. We will see in \Cref{lem.ConvexGens} that this is part of a more general phenomenon.

\begin{figure}[h!]
\begin{multicols}{2}
\begin{tikzpicture}[scale=0.10]
  \draw[->] (-2,0) -- (50,0) node[right] {$x_1$};
  \draw[->] (0,-2) -- (0,50) node[above] {$x_2$};

  \draw[thick,color1] (47,0) -- (18,1);
  \draw[thick,color2] (18,1) -- (7,3);
  \draw[thick,color3] (7,3) -- (3,8);
  \draw[thick,color4] (3,8)-- (0,47);

  \foreach \x/\y in {47/0, 18/1, 7/3, 3/8, 2/21, 1/34, 0/47} {
    \fill (\x,\y) circle (10pt);
    \node[above right] at (\x,\y) {$(\x,\y)$};
  }
\end{tikzpicture}

\begin{flushleft}
The defining ideal of $V_{(1,29),47}$ is generated by the $2\times 2$ minors of the following matrix. 
\end{flushleft}
 \[
   A= \begin{pNiceMatrix}
\Block[draw={color1}, rounded-corners]{1-2}{}
        t_1 &  \Block[draw={color2}, rounded-corners]{2-1}{} t_2 & t_3^2 & t_3t_4^2 & t_3t_4t_5 \\
        t_2^2 & \Block[draw={color3}, rounded-corners]{1-2}{}t_3 & \Block[draw={color4}, rounded-corners]{2-3}{}t_4 & t_5 & t_6 \\
        t_2t_3^2t_4 & t_4^2 & t_5 & t_6 & t_7
    \end{pNiceMatrix}.
    \]
    \begin{flushleft}
    In this matrix
the variables appear in contiguous blocks marked in four different colors, each corresponding to a segment of the convex
chain determined by $\c$. 
\end{flushleft}
\end{multicols}
\caption{Exponent vectors of the generators of $V_{(1,29),47}$ and corresponding matrix.}
\label{fig:veronese}
\end{figure}

\end{example}

We explain  how one can build from any convex sequence a matrix analogous to the matrix $A$ from~\Cref{fig:veronese} in~\Cref{const: Ac}. A condition we need for this is as follows.

\begin{definition}\label[definition]{def: cmd}
An affine semigroup ring $S\subseteq R=\kk[x_1,\dots,x_n]$ is
\emph{closed under monomial division} if whenever monomials 
$\mu_1,\mu_2\in S$ satisfy $\mu_1/\mu_2\in R$, then $\mu_1/\mu_2\in S$.
\end{definition}

\begin{lemma}\label[lemma]{lem: normal}
If a semigroup ring $S\subset \kk[x_1,\dots,x_n]$  is closed under monomial division,
then $S$ is normal.
\end{lemma}

\begin{proof}
Let $\c=\{c_1,\dots,c_s\}\subset \mathbb Z^n$ be the exponent vectors of the monomial generators of $S$.
Let $L$ be the lattice generated by $\c$ and define the cone
\[
\operatorname{pos}(\c)=\Bigl\{\sum_{i=1}^s r_i c_i \mid r_i\in \RR, r_i\ge 0\Bigr\}.
\]
Then the normalization of $S$ is
\[
\overline{S}=\kk[\{x^u\mid u\in \operatorname{pos}(\c)\cap L\}].
\]
If $x^u\in \overline{S}$, then $u=u^+-u^-$ with $x^{u^+},x^{u^-}\in S$. Since $S$ is closed
under monomial division, $x^u=x^{u^+}/x^{u^-}\in S$. Hence $\overline{S}=S$.
\end{proof}

\begin{remark}
The converse of \Cref{lem: normal} does not hold: for example,
$\kk[xy,x^2y,xy^2]$ is normal and convex, but $\frac{x^2y}{xy}=x\notin\kk[xy,x^2y,xy^2]$.
\end{remark}

\begin{construction}\label[construction]{const: Ac}
Let $\c=\{c_i\}_{i=1}^s\subset \NN^2$ be a convex sequence, and let
$\{c_{j_1},\ldots, c_{j_\ell}\}$ with $1=j_1<\cdots<j_\ell=s$ denote its corner points.
We describe the construction of a \emph{skeleton matrix} $A_\c$, by which we mean a matrix whose entries are either elements of $\kk[t_1,\ldots, t_s]$ or blank, obtained by overlapping the
matrix blocks $M_k$ from \eqref{eq: matrix M} along the main diagonal.
If $j_{k+1}=j_k+1$, so that two corners occur consecutively, we call $M_k$
\emph{degenerate} and allow it to be one of the matrices
\begin{equation}\label{eq: degenerate block}
M_k=\begin{bmatrix} t_{j_k} & t_{j_{k+1}} \end{bmatrix}
\qquad\text{or}\qquad
M_k=\begin{bmatrix} t_{j_k} \\ t_{j_{k+1}} \end{bmatrix}.
\end{equation}

The skeleton matrix $A_\c$ is obtained by placing the blocks
$M_1,\ldots,M_{\ell-1}$ along the diagonal, overlapping them at common entries.
Its entries are either variables $t_i$ or left blank.
If $M_k$ and $M_{k+1}$ are both non-degenerate, they overlap in the entry
$t_{j_{k+1}}$. For any maximal sequence of consecutive degenerate blocks
$M_k,\ldots,M_{k+t}$, we alternate horizontal and vertical placements:
$M_{k+2i}$ is $1\times 2$ and $M_{k+2i+1}$ is $2\times 1$. The outcome has the following shape with degenerate blocks in blue and nondegenerate blocks in red.
\[
 A_\c=
 \begin{pNiceMatrix}
\Block[draw={red}, rounded-corners]{2-3}{}t_1 & \cdots & t_{j_2-1} &  & & & & & \\
t_2  & \cdots & \Block[draw={red}, rounded-corners]{2-3}{} t_{j_2}  & \cdots &  t_{j_3-1} & & &  \\
   &  & t_{j_2+1}  & \cdots &\Block[draw={blue}, rounded-corners]{1-2}{} t_{j_3} & \Block[draw={blue}, rounded-corners]{2-1}{} t_{j_4} & &  \\
     &  &  & &   & t_{j_5} &\\
   & & &  & & &  \ddots\\
 & &  &  & &   & & \Block[draw={red}, rounded-corners]{2-3}{} t_{j_{\ell-1}}  & \cdots & t_{j_{\ell}-1}  \\
& &  &  &  &  & & t_{j_{\ell-1}+1}& \cdots & t_{j_{\ell}}
\end{pNiceMatrix}
\]
\end{construction}

Having built the skeleton matrix above, we now explain how to fill in the missing entries.

\begin{lemma}\label[lemma]{lem: fill Ac}
Assume that $S_\c$ is closed under monomial division in $\kk[x,y]$.
Then the blank entries of $A_\c$ can be filled with monomials of degree at least two
in $t_1,\ldots,t_{s-1}$, producing a matrix $A_\c(\z)$ such that all $2\times2$ minors
of $A_\c(\z)$ vanish after substituting $t_i$ by monomial algebra generators of $S_\c$.
\end{lemma}

\begin{proof}
Let $
c_i=(a_i,b_i)$ and $\varphi:\kk[t_1,\ldots,t_s]\to\kk[x,y]$ be given by
$\varphi(t_i)=m_i:=x^{a_i}y^{b_i}$, and let $J=\ker(\varphi)$.
We argue by induction on the number of corner points of $\c$.

If $\c$ has only two corners, then $A_\c=M_1$ has no blank entries.
By \cite[Lemma~1.2(b)]{ConvexIdeals}, the vectors $c_i$ lie evenly spaced on
$[c_1,c_s]$, and hence all $2\times2$ minors of $M_1$ vanish after substitution.

Now suppose $\c$ has $\ell>2$ corners, with indices
$j_1<\cdots<j_\ell$, and set $\c'=\{c_i\}_{i=1}^{j_{\ell-1}}$.
By induction, there exists a filled matrix $A_{\c'}(\z)$ with
$I_2(A_{\c'}(\z))\subseteq J\cap\kk[t_1,\ldots,t_{j_{\ell-1}}]$.

We extend $A_{\c'}(\z)$ by attaching the final block $M_{\ell-1}$. Depending on its shape the resulting matrix is one of the following:
\[
 \begin{pNiceMatrix}
\Block[draw={black}, rounded-corners]{4-4}{} \Block{4-4}{A_{\c'}(\z)} &  &  &  &  &** \\
\\
\\
 & 
 \mu &  & \Block[draw={red}, rounded-corners]{2-3}{} t_{j_{\ell-1}}  & \cdots & t_{j_{\ell}-1}  \\
& *  &  &  t_{j_{\ell-1}+1}& \cdots & t_{j_{\ell}}  
\end{pNiceMatrix}
\quad \text{ or } \quad
\begin{pNiceMatrix}
\Block[draw={black}, rounded-corners]{4-4}{} \Block{4-4}{A_{\c'}(\z)} &  &  \\
\\
\\
 &\mu  &  & \Block[draw={blue}, rounded-corners]{2-1}{} t_{j_{\ell-1}}   \\
& *  &  &  t_{j_{\ell}}  
\end{pNiceMatrix}
\quad \text{ or } \quad
\begin{pNiceMatrix}
\Block[draw={black}, rounded-corners]{4-4}{} \Block{4-4}{A_{\c'}(\z)} &  &  &  & ** \\
\\
\\
 & &  & \Block[draw={blue}, rounded-corners]{1-2}{} t_{j_{\ell-1}}   &  t_{j_{\ell}}
 \end{pNiceMatrix}.
\]

Consider a blank entry $*$ below an entry $\mu=\mu(t_1,\ldots,t_{j_{\ell-1}-1})$ of $A_{\c'}(\z)$.
We seek $q\in S_\c$ satisfying
\begin{equation}\label{eq: minor1-rev}
\mu(m_1,\ldots,m_{j_{\ell-1}-1})\cdot m_{j_{\ell-1}+1}
= q\cdot m_{j_{\ell-1}}.
\end{equation}
Choose $m_{i-1}\mid \mu(m_1,\ldots,m_{j_{\ell-1}-1})$ with $2\leq i\leq j_{\ell-1}$.
Convexity and \eqref{eq: convex property} give
$c_i+c_{j_{\ell-1}}<c_{i-1}+c_{j_{\ell-1}+1}$, since $c_{j_{\ell-1}}$ is a corner.
Thus $m_{j_{\ell-1}}\mid m_{i-1}m_{j_{\ell-1}+1}$ (strict divisibility), and
\[
q=\frac{\mu(m_1,\ldots,m_{j_{\ell-1}-1})m_{j_{\ell-1}+1}}{m_{j_{\ell-1}}}\in\kk[x,y].
\]
Closure of $S_\c$ under monomial division implies $q\in S_\c$, and strict divisibility implies that we may write
$q=\prod_i m_i^{\beta_i}$ with $\sum\beta_i\ge2$.
We fill $*$ with $\prod_i t_i^{\beta_i}$. All other blank entries $**$ are filled similarly.
At this point a matrix $A_\c(\z)$ has been determined, albeit not uniquely. 

It remains to check that $I_2(A_\c(\z))\subseteq J$.
Minors entirely contained in $A_{\c'}(\z)$ or $M_{\ell-1}$ lie in $J$ by induction
or the base case. 
 The other $2\times 2$ submatrices have one of the following  shapes, where $\theta, \theta'$ are entries of $A_{\c'}(\z)$ and $*, **$ and their primed versions are entries of $A_{\c}(\z)$ coming from the regions marked by these symbols above:
\[
\begin{pmatrix}
    \theta & **\\
    * & t_{k+1}
\end{pmatrix}
\text{ or }
\begin{pmatrix}
    \theta & **\\
    \theta' & t_k
\end{pmatrix}
\text{ or }
\begin{pmatrix}
    \theta & **\\
    \theta' & **'
\end{pmatrix}
\text{ or }
\begin{pmatrix}
    \theta & \theta'\\
    * & *'
\end{pmatrix}
\text{ or }
\begin{pmatrix}
    ** & **'\\
    t_i & t_j
\end{pmatrix}
\text{ or }
\begin{pmatrix}
    ** & **'\\
    **'' & **'''
\end{pmatrix}.
\]
We show that minors of the first type belong to $\ker(\varphi)$; the other cases can be argued similarly. 
From
\eqref{eq: minor1-rev}, the analogous equation arising in the construction of the monomial $q'$ in position $**$, and relations given by the minors of $A_{\c'}(\z)$ and $M_{\ell-1}$, we have
\[
\frac{q'}{m_k}=\frac{\mu'(m_1, \ldots, m_s)}{m_{j_{\ell-1}}}=\frac{\theta(m_1, \ldots, m_s)}{\mu(m_1, \ldots, m_s)}
\]
\[
\frac{m_k}{m_{k+1}}=\frac{m_{j_{\ell-1}}}{m_{j_{\ell-1}+1}}=\frac{\mu(m_1, \ldots, m_s)}{q}.
\]
Multiplying the first and last fractions in the identities displayed above gives
\[
\frac{q'}{m_{k+1}}=\frac{\theta(m_1, \ldots, m_s)}{q},
\]
which shows that desired minor belongs to $J$. 
\end{proof}

\begin{remark}\label[remark]{rem: filling}
The matrix $A_\c(\z)$ can be filled explicitly using the relations from
\cite[Lemma~1.2(b)]{ConvexIdeals}:
$\frac{m_{j+2}}{m_{j+1}}=\frac{m_{j+1}}{m_j}$ when $c_j,c_{j+1},c_{j+2}$ are collinear,
and $\frac{m_{j_k+1}}{m_{j_k}}=\frac{m_{j_k}^{q_k}}{m_{j_k-1}}$ for $q_k\ge2$ when
$c_{j_k}$ is a corner point. A fully worked out family of examples appears in \Cref{thm: 3 corner}. 
\end{remark}

To complement \Cref{thm: convex fiber ring}, we show that convex semigroup rings admit
determinantal presentations lifting those of their fiber cones. The connection is
made precise in \Cref{thm: convex semigroup} below.
Elucidating the properties of convex semigroup rings gives the main result of this section.  

\begin{theorem}\label{thm: convex semigroup}
Let $\c=\{c_i=(a_i,b_i)\}_{i=1}^s\subset \NN^2$ be a convex sequence and let $S_\c$ be the convex semigroup ring described in \eqref{eq: Sc}. Denote by $I=(x^{a_i}y^{b_i} \mid 1\leq i\leq s)$ and let $\mm=I\cap S_\c$. Assume that $S_\c$ is closed under monomial division in $\kk[x,y]$.
Then
\begin{enumerate}
\item the associated graded ring $\gr_\mm S_\c$ is isomorphic to the fiber ring $\F(I)$ as graded rings;
\item $S_\c$ has a determinantal presentation $S_\c=\kk[t_1,\ldots, t_s]/J$,
where $J$ is the ideal of $2\times 2$ minors of the matrix $A_\c$ in~\Cref{const: Ac} whose entries are nonconstant monomials in the variables $t_1, \ldots, t_s$;
\item suppose that $\deg(t_i)=d_i\in\ZZ_{>0}$, and $S_\c$ is a graded quotient, then the total Betti numbers of $S_\c$ over $Q=\kk[t_1,\ldots, t_s]$ depend only on the number $s$ of terms in the convex sequence and are given by
\[
\beta_i^Q(S_\c)= i\binom{s-1}{i+1}.
\]
\item the graded Betti numbers expressed using the notation in \eqref{eq: gap} are
\[
\beta^Q_{i,j}(S_\c)=\sum_{\substack{A\subseteq [s], |A|=i+1, \\\sum_{\ell\in A} d_\ell=j}} \gap(A)
\]
 and the Hilbert series of $S_\c$ depends only on the degrees $d_i$ and is given by \eqref{eq: HS 2dim};
\item  the minimal generators of $J$ are the $2\times 2$ minors of $A_\c$ given by submatrices whose main diagonal entries are variables. They form a Gr\"obner basis under the lexicographic order induced by $t_1>t_2>\cdots>t_n$;
\item $S_\c$ is a normal, Cohen--Macaulay, Koszul algebra;
\item the Betti numbers of the residue field $\kk$ over $S_\c$ are given by
\[
\beta_i^{S_\c}(\kk)= (s-2)^{i-2}(s-1)^2 \text{ for } i\geq 2.
\]
\end{enumerate}
\end{theorem}

\begin{proof}
(1) Let $Q=\kk[t_1,\ldots, t_s]$ and set $m_i=x^{a_i}y^{b_i}$ for each $1\leq i\leq s$. There are degree-preserving $\kk$-algebra epimorphisms $\varphi:Q\to \F(I)$ with $\varphi(t_i)=\overline{m_i}\in I/(x,y)I$ as in \eqref{eq: phi} and $\psi:Q\to \gr_\mm S_\c$ with $\psi(t_i)=\overline{m_i}\in \mm/\mm^2$. We aim to show that $\ker(\varphi)=\ker(\psi)$.

Let $F\in \ker(\varphi)$ be a homogeneous polynomial of $\deg(F)=\delta$ with respect to the standard grading on $Q$.
Then, by definition of the fiber ring \eqref{eq: fiber cone}, $F(m_1, \ldots, m_s)$ is not a minimal generator of $I^\delta$.
Thus there exist elements $g_i\in (x,y)\subset \kk[x,y]$ which may be assumed to be monomials and tuples $\alpha_i=(\alpha_{ij})_{j=1}^s$ such that for each $i$ we have $\sum_{j=1}^s \alpha_{ij}=\delta$ and
\begin{equation}\label{eq: rel}
F(m_1, \ldots, m_s)=\sum_{i=1}^r g_i \prod_{j=1}^sm_j^{\alpha_{ij}}
\end{equation}
We may also assume that the right hand side of \eqref{eq: rel} is cancellation-free so for each $1\leq i\leq r$ there must be a corresponding term $\prod_{j=1}^t t_j^{\beta_{ij}}$ in $F$ so that $g_i\prod_{j=1}^sm_j^{\alpha_{ij}}=\prod_{j=1}^t m_j^{\beta_{ij}}$. Since $S_\c$ is closed under monomial division in $\kk[x,y]$ it follows that $g_i\in S_\c$, and in fact that $g_i\in \mm$ since these elements are not constants. Thus the summation on the right hand side of \eqref{eq: rel} belongs to $\mm^{\delta+1}$ and consequently so does $F(m_1, \ldots, m_s)$. Therefore in $\gr_\mm S_\c$, we have $\psi(F)=\overline{F}=0\in \mm^{\delta}/\mm^{\delta+1}$, that is, $F\in \ker(\psi)$.

Conversely, if $F\in \ker(\psi)$, then  an equation of the form \eqref{eq: rel} is satisfied where each $g_i\in \mm$. As \eqref{eq: rel} demonstrates that $F$ is not a minimal generator of $I^\delta$, we deduce from \eqref{eq: fiber cone} that $F\in \ker(\varphi)$.

Since the maps $\varphi$ and $\psi$ have the same kernel, their images are isomorphic, that is, $\F(I)\cong \gr_\mm S_\c$.

(2) Let $A_\c$ denote the skeleton matrix in~\Cref{const: Ac}.
Denote by $A_\c(0)$ the matrix obtained by filling the blank entries of $A_\c$ with zeros and notice that the $2\times 2$ minors of $A_\c(0)$ generate precisely the ideal $L$ from \Cref{thm: convex fiber ring}.
This yields an isomorphism of graded rings $\F(I)\cong Q/I_2(A_\c(0))$.
Recall that $S_\c=Q/J$. By \Cref{lem: fill Ac}
a complete matrix $A_\c(\z)$  can be determined such that $I_2(A_\c(\z))\subseteq J$. We aim to show that this containment is in fact an equality.

Set $m_Q$ to be the homogeneous maximal ideal of $Q$. For $F\in Q$, let $F^*$ denote the image of $F$ in $\gr_{m_Q} Q$ and for an ideal $U$ of $Q$ let $U^*=\{F^* \mid F\in U\}$. Since $I_2(A_\c(\z)) \subseteq J$ it follows that $I_2(A_\c(\z))^* \subseteq J^*$. On the other hand, since the entries of $A_\c(\z)$ are at least quadratic with respect to the standard grading of $Q$ outside the skeleton blocks and linear inside these blocks, we have  $I_2(A_\c(0))\subseteq I_2(A_\c(\z))^*$. Finally, by part (1) and \Cref{thm: convex fiber ring}, $J^*=I_2(A_\c(0))$, which gives equalities
\begin{equation}\label{eq: gre equality}
\gr_{m_Q}Q/I_2(A_\c(\z))= \frac{\gr_{m_Q}Q}{ I_2(A_\c(\z))^*}=\frac{Q}{I_2(A_\c(0))}= \frac{\gr_{m_Q}Q}{J^*}=\gr_{m_Q} Q/J.
\end{equation}

Next, we proceed to show that  the equality $J=I_2(A_\c(\z))$ holds.  The containment $I_2(A_\c(\z))\subseteq J$ was shown above. Consider the grading on $Q$ given by  $\deg(t_i)=\deg(m_i)$, which makes $S_\c\cong Q/J$ a graded isomorphism. The ideal $J$ is homogeneous with respect to this grading in view of  the preceding isomorphism and we claim that its subset $I_2(A_\c(\z))$ is homogeneous as well. Indeed, the generators of $I_2(A_\c(\z))$ are binomials, so their homogeneous components are either monomials or the binomials themselves and these homogeneous components belong to $J$. However, $J$ is a prime ideal so it contains no monomials. Thus the generators of $I_2(A_\c(\z))$ are homogeneous.

Let $F\in J$ be a homogeneous element  of $m_Q$-adic order $j$ and degree $d$ with respect to the above non-standard grading on $Q$. Recall that $F^*$ denotes the image of $F$ in $\gr_{m_Q} Q$. By \eqref{eq: gre equality}, there exists an element $G_1\in  I_2(A_\c(\z))$ so that $F^*=G_1^*\in m_Q^j/m_Q^{j+1}$. Moreover, since $F$ is homogeneous of degree $d$, $G_1$ can be chosen to be homogeneous of degree $d$ as well. It follows that $F-G_1\in J$ has  $m_Q$-adic order at least $j+1$ and degree $d$. Continuing in this way, by induction on $t$, one can find a sequence of homogeneous elements $\{G_i\}_{i\geq 1}$ from $[I_2(A_\c(\z))]_d$ so that for every $t\geq 1$ we have
\[F-\sum_{i=1}^t G_i \in J_d\cap m_Q^{j+t}.\]
Since $J_d$ is a finite dimensional $\kk$-vector space, for $t\gg0$ the right hand side of the above display is zero, revealing that $F=\sum_{i=1}^t G_i\in I_2(A_\c(\z))$. This  establishes the desired equality $J=I_2(A_\c(\z))$.

(3) By (1) and \Cref{prop: betti fiber}, the Betti numbers of $\gr_\mm S_\c$ over $\gr_{m_Q} Q$ are as given in \eqref{eq: betti}. In particular, the ideal $J^*$ has  linear resolution. By \cite[Proposition 1.5]{HI} (see also \cite{RS}),
since the minimal graded free resolution of $J^*$ is linear (so $J$ is a Koszul module in the language of Herzog--Iyengar), the Betti numbers of $J^*$ and $J$, thus also those of $\gr_\mm S_\c$ and $S_\c$, coincide. 
By (1), \Cref{prop: betti fiber} can be applied to obtain the graded Betti numbers of $\gr_\mm S_\c$ and \eqref{eq: betti} yields the desired claim.

(4) We expand on the argument given in part~(3). Let $F_\bullet$ be a minimal free resolution of $S_\c$ over $Q$. This minimal free resolution is homogeneous with respect to the non-standard grading of $Q$. A complex of free $\gr_{m_Q} Q$-modules ${\rm lin}(F_\bullet)$ is obtained  by filtering $F_\bullet$ by the powers of $\mm$. Concretely, the differential of ${\rm lin}(F_\bullet)$ is obtained by retaining only the terms in $\mm\setminus \mm^2$ in the matrices representing the differential of $F_\bullet$. By part (3) the complex ${\rm lin}(F_\bullet)$ is a minimal free resolution of $\gr_\mm S_\c$ over $\gr_{m_Q} Q\cong Q$. This resolution is naturally standard graded. However, since its differentials consist of terms from the differential of  $(F_\bullet)$, the complex ${\rm lin}(F_\bullet)$ is also homogeneous and has the same graded twists as $F_\bullet$ with respect to the non-standard grading of $Q$. Thus the graded Betti numbers and the Hilbert function of $S_\c$ are the same as that of $\gr_\mm S_\c$ viewed as a quotient of $Q$ equipped with the non-standard grading $\deg(t_i)=d_i$. 

By part~(1) the invariants of interest of $S_\c$ are the same as that of $\F(I)$, which by \Cref{thm: convex fiber ring} are the same as that of its initial algebra $Q/(t_it_j \mid |i-j|>1)$ equipped with the same non-standard grading; see for example \cite[Corollary 39.7]{Peevabook}. For the latter, the Hilbert series and graded Betti numbers are determined in \eqref{eq: multiBetti} and \eqref{eq: HS 2dim}, respectively.

(5) By (3), $\beta_0(J)=\beta_0(J^*)=\binom{s-1}{2}$. To see that the $2\times 2$ minors of $A_\c(\z)$ given by submatrices whose main diagonal entries are variables are a subset of the minimal generators of $J$ it suffices to notice that their images in $Q/m_Q^3$  are distinct monomials $\{t_it_j \mid |i-j|>1\}$, thus linearly independent and therefore these minors span a vector space of dimension  $\beta_0(J)$. 

To see that these generators form a Gr\"obner basis with respect to the lexicographic order it suffices to show that their initial monomials are $\{t_it_j \mid |i-j|>1\}$, and that the Hilbert functions of $S_\c$ and $Q/(t_it_j \mid |i-j|>1)$ agree.
The former assertion follows from  the proof of \Cref{lem: fill Ac} and \Cref{rem: filling}, which show that in each $2\times 2$ minor of $A_\c(\z)$ which has variables $t_i, t_j$ with $i<j$ along the main diagonal, the support of the entries not along the main diagonal involves only variables $t_k$ with $i<k$. The latter assertion follows by comparing item (3) and \Cref{prop:closedformulamonomial}.

(6) By \Cref{lem: normal}, $S_\c$ is normal. Since by part (1) and \Cref{thm: convex fiber ring} the ring $\gr_\mm S_\c$ is Cohen-Macaulay and Koszul; the same is true for $S_\c$, in the case of the Cohen-Macaulay property by \cite[Proposition 25.4]{Peevabook} and by definition in the Koszul case.

(7) By part~(6), $S_\c$ is normal. The given formula applies for Betti numbers of $\kk$ over two-dimensional normal semigroup rings by \cite[Theorem 3.9]{PRS}.
\end{proof}

\begin{example}
    Consider the Veronese ring $V_{(1,29),47}$ as in \Cref{ex: convex}, a convex semigroup ring defined by the convex sequence 
    $$\{c_1=(0,47), c_2=(1,34), c_3=(2,21), c_4=(3,8), c_5=(7,3), c_6=(18,1), c_7=(47,0)\}.$$
    This is the quotient of a non-standard graded polynomial ring $Q=\kk[t_1,\dots, t_7]$ with
    $$\deg(t_1)=1,\quad\deg(t_2)=1,\quad\deg(t_3)=2,\quad\deg(t_4)=5,$$
    $$\deg(t_5)=13,\quad\deg(t_6)=21,\quad\deg(t_7)=29,$$
    by the determinantal ideal $J$ of $2$-minors of the matrix $A$ (see Figure \ref{fig:veronese}). We remark that the degrees of the variables $t_i$ have been divided by $47$. The total Betti numbers of $V_{(1,29),47}$ over $Q$ are $(1,\,15,\,40,\,45,\,24,\,5)$. In particular, we compute some of the graded Betti numbers $\beta_{i,j}:=\beta_{i,j}^Q(V_{(1,29),47})$ using the formula in \Cref{thm: convex semigroup} (4).
    \begin{align*}
        \beta_{5,71}&=\gap(\{1,3,4,5,6,7\})+\gap(\{2,3,4,5,6,7\})=1+0=1,\\
        \beta_{3,37}&=\gap(\{1,3,4,7\})+\gap(\{2,3,4,7\})+\gap(\{1,3,5,6\})+\gap(\{2,3,5,6\})\\
        &=2+1+2+1=6,\\
        \beta_{3,9\,\,}&=\gap(\{1,2,3,4\})=0,\\
        \beta_{2,27}&=\gap(\{1,4,6\})+\gap(\{2,4,6\})=2+2=4.
    \end{align*}
    
    We emphasize that the only possible non-zero internal degrees different are sums of degrees of at least two non-consecutive variables. This quantity contributes to the homological degree one less than the number of variables involved in the sum. The numerical invariants above have been checked via \texttt{Macaulay2} \cite{M2}.
    
    The ring $V_{(1,29),47}$ is normal, Cohen-Macaulay, and Koszul. The base field $\kk$ has a minimal (infinite) free resolution over $S_\c$ with Betti numbers $\beta_i^{S_\c}(\kk)=36\cdot 5^{i-2}$.
\end{example}

\section{Two-dimensional  Veronese rings}\label{s: 2dim}

An affine semigroup ring is a subring of $\kk[x_1^{\pm1}, \ldots, x_n^{\pm 1}]$ generated by Laurent monomials. Normal affine semigroup rings are the coordinate rings of affine normal toric varieties. In this section we discuss  two-dimensional normal affine semigroup rings observing that the bulk of their properties reduces to those of weighted Veronese rings. This is essentially a consequence of the fact that all two-dimensional cones are simplicial. 

We will prove the following implications in \Cref{prop: 2Dnormal_is_Veronese} and \Cref{lem.ConvexGens}:
\begin{quote}
    $S=$ 2-dimensional normal affine semigroup ring
    $\Longrightarrow$
    $S\cong V_{(1,c),d}$ for $c<d$ \\ \ $\Longrightarrow$ $S$ is isomorphic (non graded isomorphism) to a convex semigroup.
\end{quote}
This will allow us to apply those results in \Cref{thm: convex semigroup} which are independent of grading to all two-dimensional normal semigroup rings.

The following result is an algebraic restatement of the fact that affine toric surfaces have cyclic quotient singularities \cite[\S2.2]{Fulton}. See \cite[Proposition 10.1.1]{CLS} for a version of this result that does not mention Veronese rings.

\begin{proposition}\label[proposition]{prop: 2Dnormal_is_Veronese}
   Any two-dimensional normal  affine semigroup ring is isomorphic to a weighted Veronese ring $V_{(1,c),d}$ with $0\le c<d$, and $\gcd(c,d)=1$ by a not necessarily degree-preserving isomorphism.
\end{proposition}
\begin{proof}
Let $S$ be a normal affine semigroup ring of dimension $2$.  Then there exists  a strongly convex rational cone $C$ with primitive generators
$u,v\in\ZZ^2$ and $S=\kk[C\cap\ZZ^2]$. After a unimodular change of coordinates, we may assume
$u=(1,0)$ and $v=(c,d)$ with $d\ge1$, $0\le c<d$, and $\gcd(c,d)=1$. The inequality $0\le c<d$ can be achieved by reducing the first coordinate of the image of $v$ modulo the second using the Euclidean algorithm and primitivity forces $\gcd(c,d)=1$.
Thus
\[
C\cap \ZZ^2=\{(i,j)\in\NN^2\mid di-cj\ge0, i\geq 0, j\geq 0\}.
\]
The desired isomorphism comes from the lattice (injective) homomorphism mapping $C$ to the positive orthant by  $(i,j)\mapsto (di-cj,j)$. The lattice points $(a,b)=(di-cj,j)$ in the image satisfy $a+cb=di$ for integers $i\geq 0$, which gives the semigroup isomorphism, and therefore, the desired isomorphism
\[
S\cong \kk[x^{di-cj}y^j\mid (i,j)\in C\cap \ZZ^2]\cong V_{(1,c),d}. \qedhere
\]
\end{proof}

We now turn to the property of a two-dimensional affine semigroup ring of being convex. Our main examples of convex semigroup rings are the two-dimensional Veronese rings.
The next result shows that they fit into the framework developed in~\Cref{s: convex semigroups}.

\begin{proposition}\label[proposition]{lem.ConvexGens}
Two-dimensional Veronese rings $V_{(a,b),d}$ are convex semigroup rings.
\end{proposition}

\begin{proof}
Let $m_i=x^{a_i}y^{b_i}$ be a minimal set of $\kk$-algebra generators of
$V_{(a,b),d}$, and order them so that the exponent of $y$ decreases, i.e.
$b_{i+1}\leq b_{i}$ for $i=1,\ldots,s-1$.
Then
\[
\left\{\left(a_{i},b_i,\frac{a \,a_i+b \,b_i}{d}\right)\mid i=1,\ldots,s \right\}
\]
is the set of \emph{minimal} nonnegative solutions of the linear Diophantine equation
\begin{equation}\label{eq: diophantine}
ax+by=dz.
\end{equation}
We record two immediate consequences of minimality:
\begin{itemize}
\item $b_i\neq b_j$ for $i\neq j$. If $b_i=b_j$, then
$\left(|a_i-a_j|,0,\frac{a|a_i-a_j|}{d}\right)$ is a nonzero solution of
\eqref{eq: diophantine}, so $m_i$ and $m_j$ cannot both be minimal algebra generators.
\item If $b_i>b_j$, then $a_i<a_j$. Otherwise, if $a_i\ge a_j$, then
\[
\left(a_i-a_j,\, b_i-b_j,\,
\frac{a(a_i-a_j)+b(b_i-b_j)}{d}\right)
\]
is a nonzero solution to \eqref{eq: diophantine}, contradicting the minimality of $m_i$.
\end{itemize}
Hence, after reindexing if necessary, we may write $c_i=(a_i,b_i)\in \NN^2$ and assume
\[
b_1 > b_2> \cdots > b_s 
\qquad\text{and}\qquad
a_1< a_2< \cdots <a_s.
\]
We claim that $\{c_i\}_{i=1}^s$ is convex.
Suppose not, and fix $i$ with $2c_i\not\le c_{i-1}+c_{i+1}$.

 \emph{Case A:} $2c_i>c_{i-1}+c_{i+1}$ componentwise.
Set
\[
k=(k_1,k_2):=2c_i-c_{i-1}-c_{i+1},\qquad t=(t_1,t_2):=c_{i-1}+c_{i+1}-c_i.
\]
Then $k,t\in \NN^2$ and, using the strict monotonicity of the coordinates,
\begin{align*}
0<k_1&=2a_i-a_{i-1}-a_{i+1}<a_{i}-a_{i-1}\le a_{i},\\
0<k_2&=2b_{i}-b_{i-1}-b_{i+1}<b_{i}-b_{i+1}<b_{i},
\end{align*}
and
\begin{align*}
0\le a_{i-1}<t_1&=a_{i-1}+a_{i+1}-a_{i}<a_{i},\\
0\le b_{i+1}<t_2&=b_{i-1}+b_{i+1}-b_{i}<b_{i}.
\end{align*}
Since the triples corresponding to $c_{i-1},c_i,c_{i+1}$ solve \eqref{eq: diophantine},
it follows that both
\[
\left(k_1,k_2,\frac{ak_1+bk_2}{d}\right)
\quad\text{and}\quad
\left(t_1,t_2,\frac{at_1+bt_2}{d}\right)
\]
are nonnegative solutions of \eqref{eq: diophantine}, and their sum equals
$\left(c_{1,i},c_{2,i},\frac{ac_{1,i}+bc_{2,i}}{d}\right)$.
This contradicts the minimality of $m_i$.

\emph{Case B:} exactly one component violates convexity.
Without loss of generality, assume
\[
2a_{i}>a_{i-1}+a_{i+1}
\quad\text{and}\quad
2b_{i}\le b_{i-1}+b_{i+1}.
\]
Set again $t:=c_{i-1}+c_{i+1}-c_i$. Then
\[
a_{i-1}<t_1=a_{i-1}+a_{i+1}-a_{i}<a_{i},
\qquad
b_{i}\le t_2=b_{i-1}+b_{i+1}-b_{i}<b_{i-1}.
\]
Thus $\left(t_1,t_2,\frac{at_1+bt_2}{d}\right)$ is a nonnegative solution of
\eqref{eq: diophantine}. It cannot be minimal since $m_{i-1}$ and $m_i$ are consecutive
minimal solutions, so we may write it as a sum of two nonzero
solutions
\[
(u_1,u_2, \tfrac{au_1+bu_2}{d})+(v_1,v_2, \tfrac{av_1+bv_2}{d})
=(t_1,t_2, \tfrac{at_1+bt_2}{d}).
\]
Choose such a decomposition with the first summand minimal, corresponding
to a generator $m_j=x^{u_1}y^{u_2}$.
Since $u_1<t_1<a_{i}$, we have $j<i$; and since $u_2<t_2<b_{i-1}$ while the $y$-exponents
strictly decrease with the index, we must have $j>i-1$. This produces a contradiction.
\end{proof}

A first consequence of convexity is a closed formula for the largest degree of the
minimal semigroup generators of two-dimensional Veronese rings.

\begin{corollary}\label[corollary]{thm.maxdeg2var}
Suppose that $a,b\in \NN$ are coprime. Then the maximum degree of minimal generators of
$V_{(a,b),d}$, denoted $w(V_{(a,b),d})$, is given by
\[
w(V_{(a,b),d})=\max\{\ell_1,\ell_2\},\qquad \ell_1=\lcm(a,d),\ \ \ell_2=\lcm(b,d).
\]
\end{corollary}

\begin{proof}
This follows from the convexity in \Cref{lem.ConvexGens}. Set $\ell=\max\{\ell_1,\ell_2\}$
and $\ell'=\min\{\ell_1,\ell_2\}$, and consider the line $ax+by=\ell$.
Both $(\ell/a,0)$ and $(0,\ell/b)$ are lattice points on this line, and $\ell'$ lies on
or below it. By convexity, all other exponent vectors $c_i$ lie on or below the same
line, hence $ac_{1,i}+bc_{2,i}\le \ell$ for all $i$. Therefore every minimal generator
has degree at most $\ell$.
\end{proof}

\begin{remark}
Convexity does not hold for Veronese rings of dimension $n\ge 3$. For example,
\[
V_{(1,6,15),10}
=\kk[x_{1}^{10},x_{1}^{4}x_{2},x_{1}^{5}x_{3},x_{1}^{2}x_{2}^{3},
x_{1}^{3}x_{2}^{2}x_{3},x_{2}^{5},x_{3}^{2},x_{1}x_{2}^{4}x_{3}],
\]
the exponent vectors of the monomial generators fail to form a convex sequence: the
lattice points $(5,0,1)$, $(2,3,0)$, and $(3,2,1)$ lie in the convex hull of the other
exponent vectors. This example also illustrates that an analogous formula to \Cref{thm.maxdeg2var} cannot hold in dimension greater than two, as for the ring above the maximum degree of an algebra generator is $40>30=\max\{\lcm(1,10),\lcm(6,10),\lcm(15,10)\}$.
\end{remark}

We now arrive at our main theorem.

\begin{theorem}\label{thm: 2-dimVeronese}
Two-dimensional normal affine semigroup rings $S$ are Cohen--Macaulay, Koszul, and admit a determinantal presentation in which the relations are given by the ideal of $2\times 2$ minors of a matrix  whose entries are nonconstant monomials. Their minimal relations form a Gr\"obner basis. Their  total Betti numbers, and the Betti numbers of $\kk$ over these rings depend only on the number $s$ of monomial generators of $S$ as a $\kk$-algebra as follows:
\begin{enumerate}
\item $\beta_i^R(S)= i\binom{s-1}{i+1}$ for all $i\ge 1$.
\item \cite{PRS} $\beta_i^{S}(\kk)= (s-2)^{i-2}(s-1)^2$ for all $i\ge 2$.
\end{enumerate}
The graded invariants of two-dimensional Veronese rings additionally depend 
on the degrees $\{d_i\}_{i=1}^s$ of their $\kk$-algebra generators,
as follows:
\begin{enumerate}
\item[(3)] $HS(z,V_{(a,b),d}) = \dfrac{1 + \sum_{A \subseteq [s]}(-1)^{|A| + 1}\gap(A)\, z^{\sum_{k \in A} d_k}}{\prod_{j = 1}^s (1 - z^{d_j})}$.
\item[(4)] $\beta_{i,j}(S_\c)=\sum_{A\subseteq [s], |A|=i+1, \sum_{\ell\in A} d_\ell=j} \gap(A).$
\end{enumerate}
\end{theorem}

\begin{proof}
By \Cref{prop: 2Dnormal_is_Veronese} there is an isomorphism $S\cong V_{(1,c),d}$ for some integers $c,d$. Since it is a ring homomorphism, it preserves the number of minimal monomial algebra generators. Since the properties preceding and including (1) and (2) are isomorphism invariant, the claim follows from \Cref{thm: convex semigroup} once we verify its hypotheses apply to two-dimensional Veronese rings.  By \Cref{lem.ConvexGens} any Veronese ring $V_{(a,b),d}$ is a convex semigroup ring. It remains to show that it is closed under monomial division in $\kk[x,y]$.
Let $\mu_1,\mu_2$ be monomials in $V_{(a,b),d}$ such that the quotient $\mu_1/\mu_2\in \kk[x,y]$. Then the degree of $\mu_1/\mu_2$ is a nonnegative multiple of $d$, hence $\mu_1/\mu_2\in V_{(a,b),d}$, establishing closure under monomial division.
\end{proof}

\begin{remark}
    \label{rem: normal higher dim}
    By the same reasoning as in the above theorem, all weighted Veronese rings are closed under monomial division, hence, are normal by \Cref{lem: normal}, and therefore, are Cohen-Macaulay by \cite{Hoc}.
\end{remark}

\subsection{Two segment Veronese rings}\label{s: 3 corners}

In the standard graded case, a two-dimensional Veronese ring corresponds to a semigroup
whose generators lie on a single line segment. 
Here we classify the next simplest case, namely those two-dimensional
Veronese semigroups $V_{(1,c),d}$, with $c,d$ coprime, whose generators lie on exactly two
line segments. This subsumes the general case by \Cref{prop: 2Dnormal_is_Veronese}. As an application of \Cref{thm: convex semigroup}, we obtain an explicit
determinantal presentations in this case in \Cref{thm: 3 corner}.

The next result extends \cite[Theorem 6.2]{Sum}, which treats the case where $d$ is prime. When $d=p$ is prime, the condition $p=cm+r$ with $r\mid(c-1)$ is equivalent to the formulation in loc.cit.~involving the
inverse of $c$ modulo $p$.

\begin{proposition}\label[proposition]{thm:iff2lines}
Let $c,d\in\NN$ be coprime and write $d=cm+r$ with $0<r<c$. Then $r\mid(c-1)$ if and only if the algebra generators of
the convex semigroup $V_{(1,c),d}$ 
 lie on at most two line segments.
\end{proposition}

\begin{proof}
Let  $q=\frac{c-1}{r}$. Assume first that $r\mid(c-1)$ so that $q$ is an integer. Define
\[
S_1=\{x^{(m-i)c+r}y^i \mid i=0,\dots,m\},\qquad
S_2=\{x^j y^{d-(qm+1)j} \mid j=0,\dots,r-1\}.
\]
The points corresponding to $S_1$ lie on the line $x+cy=d$, while those corresponding to
$S_2$ lie on the line $y=-(qm+1)x+d$. The line segments between $(r,m)$ and $(d,0)$ on the first line and between $(0,d)$ and $(r,m)$ on the second line contain as integer points exactly the exponents listed above.

We claim that $S_1\cup S_2$ minimally generates $V_{(1,c),d}$ as $\kk$-algebra.
Suppose the contrary that $x^j y^i\notin S_1\cup S_2$ is a minimal generator. If $i\le m$,
then $x^{(m-i)c+r}y^i\in S_1$, forcing $d\mid |(m-i)c+r-j|$. Since $(m-i)c+r\leq  mc+r=d$, we must have $j=(m-i)c+r +kd$, for some $k>0$.
Hence
$x^j y^i=(x^d)^k x^{(m-i)c+r}y^i$, contradicting minimality. A symmetric argument excludes
$j\le r$, and if $i\ge m$ and $j\ge r$, then $x^{j-r}y^{i-m}\in V_{(1,c),d}$, so
$x^j y^i$ is not minimal. Hence $S_1\cup S_2$ is precisely the minimal generating set.

Conversely, assume that all minimal generators of $V_{(1,c),d}$ lie  on two lines. Since $x^d,x^ry^m$, and $y^d$ are minimal generators, by convexity, the equations of the two lines are given by $x+cy=d$ and $y=-(qm+1)x+d$.
First, we note that there exists some $b \in \mathbb N$ such that $xy^b$ is a minimal $\kk$-algebra generator for any $V_{(1,c),d}$. Indeed, since $c$ is invertible modulo $d$, there exists $b\in\{1,\dots,d-1\}$ such that
$1+bc=dk$.
When $k=1$, then the generator $xy^b$ lies on the first line, and one has the relation $1+bc=cm+r$. This last relation requires $r=1$ since $r$ is the remainder modulo $c$, so it follows immediately that $r|(c-1)$. On the other hand, if $k>1$, $xy^b$ must lie on the second line, which forces
\[
b=d-(mq+1)=\frac{d(rc-c+1)-r}{rc}.
\]
Substituting into $1+bc=dk$ yields
\[
k=c-\frac{c-1}{r},
\]
which is an integer only if $r\mid(c-1)$.
\end{proof}

The following result illustrates \Cref{const: Ac} and \Cref{lem: fill Ac} in a situation where the matrix therein can be described explicitly.

\begin{proposition}\label[proposition]{thm: 3 corner}
Let $d=cm+r$ with $r\mid(c-1)$. Set $q=\frac{c-1}{r}$. Then:
\begin{enumerate}
\item there is a surjective morphism $\kk[t_1,\dots,t_{m+r+1}]\twoheadrightarrow V_{(1,c),d}$, where
\[
t_{i+1}\mapsto  x^{(m-i)c+r}y^i \quad (i=0,\dots,m),\qquad
t_{m+r+1-j} \mapsto x^j y^{d-(mq+1)j} \quad (j=0,\dots,r-1).
\]
\item The defining ideal is generated by the $2\times 2$ minors of $M$, where
$$M = \begin{bmatrix}
            t_1 & t_2 & \cdots & t_{m-1} & t_m & t_{m+1}^{q + 1} & t_{m+1}t_{m+2} & t_{m+1}t_{m+3} & \cdots & t_{m+1}t_{m+r-1}\\
            t_2 & t_3 & \cdots & t_m & t_{m+1} & t_{m+2} & t_{m+3} & t_{m+4} & \cdots & t_{m+r} \\
            t_3t_{m+1}^q & t_4t_{m+1}^q & \cdots & t_{m+1}^{q +1} & t_{m+2} & t_{m+3} & t_{m+4} & t_{m+5} & \cdots & t_{m+r+1}
        \end{bmatrix}$$
\end{enumerate}
\end{proposition}

\begin{proof}
Part~(1) follows from the description of generators in the proof of
\Cref{thm:iff2lines}. Part~(2) follows from~\Cref{const: Ac} and
\Cref{thm: convex semigroup}. The skeleton matrix is filled using the following relations from \Cref{rem: filling}, where $m_i$ is the monomial in $\kk[x,y]$  corresponding to $t_i$ 
\[
\frac{m_{j+2}}{m_{j+1}}=\frac{m_{j+1}}{m_j}
\quad\text{for collinear triples } (j\neq m), 
\qquad
\frac{m_{m+2}}{m_{m+1}}=\frac{m_{m+1}^{q+1}}{m_{m}}
\quad\text{at the corner.}\qedhere
\]
\end{proof}


\section{Higher-dimensional Veronese rings}\label{s: higher dimensional Veronese rings}

In this section we illustrate several ways in which Veronese rings of dimension at least
three differ from their two-dimensional counterparts.

\subsection{Failure of determinantal presentations}

Unlike the two-dimensional case, higher-dimensional Veronese rings need not admit
determinantal presentations.

\begin{proposition}\label[proposition]{prop: V_{(3,4,5),6}}
The ring $V_{(3,4,5),6}$ does not admit a determinantal presentation.
\end{proposition}

\begin{proof}
A  computation in \texttt{Macaulay2} \cite{M2} shows that
\[
V_{(3,4,5),6}
=\kk[x_1^2,x_1x_2x_3,x_2^3,x_1x_3^3,x_2^2x_3^2,x_2x_3^4,x_3^6].
\]
Let $Q=\kk[t_1,\ldots,t_7]$ and map $t_i$ to the generators above, in the given order.
The kernel $I$ of this map is generated by the quadratic binomials
\begin{multline*}
I=(t_2^2-t_1t_5,\ t_3t_4-t_2t_5,\ t_2t_4-t_1t_6,\ t_5^2-t_3t_6,\ t_4t_5-t_2t_6,\\
t_4^2-t_1t_7,\ t_5t_6-t_3t_7,\ t_4t_6-t_2t_7,\ t_6^2-t_5t_7).
\end{multline*}

Observe that: (1)~the monomial $t_2t_3$ does not appear in any generator of $I$; (2)~
 $I$ is generated by quadratic binomials, each quadratic monomial appearing in at
most one generator; (3)~$I$ is prime, hence none of these binomials is reducible.

Suppose $I=I_2(M)$ for some matrix $M$ with nonconstant monomial entries.
Since $t_2^2-t_1t_5\in I$, by~(2)~we may assume (up to row and column permutations) that
\[
M=\begin{pNiceMatrix}
t_2 & t_1  & \Block{2-2}<\Large>{A}  \\
t_5 & t_2 & & \\

\Block{2-2}<\Large>{B} & &  \Block{2-2}<\Large>{C}\\
\\
\end{pNiceMatrix}.
\]
The relation $t_5^2-t_3t_6\in I$ then forces $t_3$ to appear as an entry of $M$. We show that some nonzero $2\times 2$ minor of $M$ must involve $t_2t_3$, contradicting item~(1). We have three cases in each of which the given minor(s) involve the monomial $t_2t_3$.

 If $t_3$ is one of the entries of $A$, then we have one of the following two minors in $M$
        $$
            \begin{bmatrix}
               t_2 & \ast \\
               t_5 & t_3
            \end{bmatrix}
        \quad \mbox{or} \quad \begin{bmatrix}
            t_1 & t_3 \\
            t_2 & \ast
        \end{bmatrix}.
        $$

Similarly, if $t_3$ is an entry of $B$, we have one of the two minors in $M$
        $$
            \begin{bmatrix}
               t_2 & t_1 \\
               \ast & t_3
            \end{bmatrix}
        \quad \mbox{or} \quad \begin{bmatrix}
            t_5 & t_2 \\
            t_3 & \ast
        \end{bmatrix}.
        $$

Finally, if $t_3$ is an entry of $C$, then taking the first row and column, and the row and column where $t_3$ appears, gives the following minor of $M$
        $$
            \begin{bmatrix}
                t_2 & a \\
                b & t_3
            \end{bmatrix}.
        $$
 Having obtained a contradiction in each case we conclude that $I$ cannot be generated by the $2\times 2$ minors of $M$.       
\end{proof}

\subsection{Failure of the Koszul property}

Higher-dimensional Veronese rings may also fail to be Koszul.

\begin{example}\label[example]{ex: notKoszul}
The ring $V_{(3,4,5),15}$ is not Koszul.
\end{example}

A computation in \texttt{Macaulay2} \cite{M2} yields
\begin{equation}\label{eq: nonKoszul}
V_{(3,4,5),15}
=\kk[x_1^5,x_1^2x_2x_3,x_1x_2^3,x_3^3,x_2^5x_3^2,x_2^{10}x_3,x_2^{15}].
\end{equation}
Let $Q=\kk[t_1,\ldots,t_7]$ map onto $V_{(3,4,5),15}$ via the generators in
\eqref{eq: nonKoszul}. The kernel contains the cubic relation
\[
t_2^3-t_1t_3t_4,
\]
which survives in the associated graded ring $\gr_\mm(Q/I)$. Since this relation is
minimal and non-quadratic, the ring is not Koszul.

The previous example is the first instance of two families of non-Koszul weighted Veronese rings.

\begin{proposition}\label[proposition]{prop: nonKoszul}
The following families of Veronese rings are not Koszul:
\begin{enumerate}
    \item $V_{(3,3k+1,3k+2,w_4,\ldots,w_n),\,3(3k+2)}$ for any integers $k\geq 1$ and $w_4,\ldots,w_n>3(3k+2)$;
    \item $V_{(f+1,f+2,f^2+f-1,w_4,\ldots,w_n),\,3(f^2+f-1)}$ for any integers $f\geq 2$ and $w_4,\ldots,w_n>3(f^2+f-1)$.
\end{enumerate}
\end{proposition}

\begin{proof}
The two families exhibit similar behavior resulting in a cubic generator for the defining ideal of the respective Veronese ring. We give a detailed proof for the more involved case (2). Case (1) is established by an analogous argument.


(2) Let $R=\kk[x_1,\ldots,x_n]$ with
\[
\deg(x_1)=f+1,\quad \deg(x_2)=f+2,\quad \deg(x_3)=f^2+f-1,\quad
\deg(x_i)=w_i>3(f^2+f-1)\ (i\ge4).
\]
Degree considerations show that any monomial of degree $3(f^2+f-1)$ involves only
$x_1,x_2,x_3$. Solving
\[
a(f+1)+b(f+2)+c(f^2+f-1)=3(f^2+f-1)
\]
with $0\le c\le3$ yields exactly four solutions, corresponding to the monomials
\[
x_1^fx_2^{f-1}x_3,\quad x_1^{2f+1}x_2^{f-2},\quad
x_1^{f-1}x_2^{2f-1},\quad x_3^3.
\]
For example, if $c=0$ the equation becomes
  \[
a(f+1)+b(f+2)=3(f^2+f-1)=3f(f+1)-3
\]
  which implies $b\equiv -3 \pmod{f+1}$. As $0\leq b<3f$ we obtain the possibilities $b\in\{f-2, 2f-1\}$. Solving for $a$ one obtains the second and third monomials listed above. The analysis in the remaining cases is similar, thus we omit it.

Mapping $Q=\kk[t_1, t_2, t_3, t_4, \ldots]$ onto the Veronese ring via these generators produces the
relation $t_1^3-t_2t_3t_4$. No quadratic relations exist among $t_1,\ldots,t_4$, and any
relation involving the remaining variables has strictly larger degree. As it is homogeneous with respect to the standard grading and there are no relations of the weighted Veronese ring whose initial terms are quadratic and involve only $t_1, t_2, t_3, t_4$ by the previous considerations,
$t_1^3-t_2t_3t_4$ is also a minimal relation of the associated graded ring of the given Veronese, which therefore is
not Koszul.
\end{proof}


\bibliographystyle{amsalpha}
\bibliography{biblio}

\bigskip


\end{document}